\documentclass{article}
\usepackage[margin=1in]{geometry}
\usepackage{amsmath,array}
\usepackage{amssymb}
\usepackage{amsthm}
\usepackage[T1]{fontenc}
\usepackage{pvscript}
\usepackage[pdfauthor={your name},
            pdftitle={The Title},
            pdfsubject={The Subject},
            pdfkeywords={Some Keywords},
            pdfproducer={Latex with hyperref},
            pdfcreator={latex->dvips->ps2pdf},
            pdfpagemode=UseOutlines,
            bookmarksopen=true,
            letterpaper,
            bookmarksnumbered=true]{hyperref}
\usepackage{memhfixc}

\usepackage{memhfixc,amssymb,amsthm,amsmath,amsfonts,mathrsfs,graphics,graphicx,lineno,hyperref}

\newtheorem{theorem}{Theorem}[section]
\newtheorem{prop}{Proposition}[section]
\newtheorem{lemma}{Lemma}[section]
\newtheorem{corollary}{Corollary}[section]
\numberwithin{equation}{section}

\begin{document}

\begin{center}
{\Large {\bf Convergence of Dirichlet Eigenvalues for Elliptic Systems
on Perturbed Domains}}\\ \vspace{.2in} Justin L. Taylor\\
Department of Mathematics and Statistics\\ Murray State University \\ Murray,
Kentucky, U.S.A.\\ jtaylor52@murraystate.edu
\end{center}

\begin{center}
\begin{tabular}{m{11cm}}
\begin{center}{\small Abstract}\end{center}

 {\footnotesize

 We consider the eigenvalues of an elliptic operator
\begin{center}$(Lu)^{\beta}=-\frac{\partial}{\partial
x_j}\left(a^{\alpha \beta}_{ij}\frac{\partial u^{\alpha}}{\partial
x_i}\right)\hspace{.5in} \beta=1,...,m$\end{center} where
$u=(u^1,...,u^m)^t$ is a vector valued function and $a^{\alpha
\beta}(x)$ are $(n \times n)$ matrices whose elements $a^{\alpha
\beta}_{ij}(x)$ are at least uniformly bounded measurable
real-valued functions such that
\begin{center}$a^{\alpha \beta}_{ij}(x)=a^{\beta \alpha}_{ji}(x)$\end{center} for any
combination of $\alpha, \beta, i,$ and $j$. We assume we have two
non-empty, open, disjoint, and bounded sets, $\Omega$ and
$\widetilde{\Omega}$, in $\mathbb{R}^n$, and add a set
$T_{\varepsilon}$ of small measure to form the domain
$\Omega_{\varepsilon}$. Then we show that as $\varepsilon
\rightarrow 0^+$, the Dirichlet eigenvalues corresponding to the
family of domains $\{\Omega_{\varepsilon}\}_{\varepsilon>0}$
converge to the Dirichlet eigenvalues corresponding to
$\Omega_0=\Omega \cup \widetilde{\Omega}$. Moreover, our rate of
convergence is independent of the eigenvalues. In this paper, we
consider the Lam\'{e} system, systems which satisfy a strong
ellipticity condition, and systems which satisfy a Legendre-Hadamard
ellipticity condition.}
\end{tabular}
\end{center} 
\vspace{.3in}

\section{Introduction}

\indent

There is a great deal of work studying eigenvalues for elliptic
equations, but there seems to be less work on eigenvalues for
elliptic systems. Much of the work on equations requires estimates
for solutions that do not hold for systems. In this paper, we consider the behavior of
eigenvalues for elliptic systems in singularly perturbed domains. We
give a simple characterization of the families of domains that we
can study and this class includes families such as dumbbell
domains formed by connecting two domains by a thin tube. We show that as the measure of the perturbation
shrinks away, the convergence of the eigenvalues is obtained. We
also provide a rate of convergence, which is independent of any
eigenvalue. We make no assumption on the smoothness of the
coefficients and only mild assumptions on the boundary of the
domain.

Studying solutions of elliptic boundary value problems with
Dirichlet or Neumann boundary conditions on domains which can be approximated
by solutions on simpler domains has been an interest for many years, and is still ongoing. The motivation to study such problems is that it is easier to study the spectra on sets with a reduced dimensionality. One may approximate the spectra on these ``fattened'' sets with the spectra on the ``thinner'' sets. Some applications include studying quantum wires, free-electron theory of conjugated molecules, and photonic crystals. For a complete description, see the work of Kuchment \cite{Kuchment}. Recent work by Exner and Post \cite{Exner} study the Neumann Laplacian on manifolds with thin tubes which is related to the theory of quantum graphs.  The Fireman's Pole problem consists of approximating
the resolvents of a bounded set in $\mathbb{R}^3$ by the resolvents
of this set with a cylinder removed. For a complete description, see
Rauch and Taylor \cite{Rauch}. A classic paper by Babuska and
V\'{y}born\'{y} \cite{Babuska} shows continuity of Dirichlet
eigenvalues for elliptic equations under a regular variation of the
domain, but gives no rates of convergence. Dancer \cite{Dancer1},
\cite{Dancer2} considers how perturbing the domain affects the
number of positive solutions for nonlinear equations with Dirichlet
boundary conditions and includes the case where solutions are
eigenfunctions for the Laplacian. Davies \cite{Davies} and Pang
\cite{Pang} study the approximation of Dirichlet eigenvalues and
corresponding eigenfunctions in a domain $\Omega$ by eigenvalues and
eigenfunctions in sets of the form $R(\varepsilon)=\{x\in
\Omega:\textrm{dist}(x,\partial \Omega)\geq \varepsilon\}$. They
each give rates of convergence and their estimates include the case
when the domain is irregularly shaped. The work of Brown, Hislop,
and Martinez \cite{Brown2} provides upper and lower bounds on the
splitting between the first two Dirichlet eigenvalues in a symmetric
dumbbell region with a straight tube. Chavel and Feldman \cite{Chavel} examine eigenvalues on a compact manifold with a small handle and Dirichlet conditions on the ends of the handle. The work of Ann\'{e} and
Colbois \cite{Anne1} examines the behavior of eigenvalues of the
Laplacian on $p$-forms under a singular perturbation obtained by
adding a thin handle to a compact manifold, but requires more
regularity on the eigenfunctions than holds in our setting.

More recent work for Dirichlet conditions includes work by Daners \cite{Daners}, which shows convergence of solutions
to elliptic equations on
sequences of domains. These domains $\Omega_n$ converge to a limit
domain $\Omega$ in the sense of sequences $u_n\in H^1_0(\Omega_n)$
converging to a function $u\in H^1_0(\Omega)$. Also, Burenkov and
Lamberti \cite{Burenkov} prove sharp spectral stability estimates
for higher-order elliptic operators on domains in certain H\"{o}lder
classes in terms of the Lebesgue measure of the symmetric difference
of the different domains. Kozlov \cite{Kozlov} obtains asymptotics of Dirichlet eigenvalues for domains in $\mathbb{R}^n$ for $n \geq 2$ using Hadamard's formula. Grieser and Jerison \cite{Grieser} also give asymptotics for Dirichlet eigenvalues and eigenfunctions, but only on plane domains.

We note here that the results for Neumann eigenvalues may be
different than those for Dirichlet eigenvalues. In fact, a classic
example of Courant and Hilbert $\cite{Courant}$ shows that the
Neumann eigenvalues may not vary continuously as the domain varies.
Their example is constructed by taking the unit square in
$\mathbb{R}^2$ and attaching a thin handle with a proportional
square attached to the other end. They show that if
$\{\lambda_n^{\varepsilon}\}$ and $\{\lambda_n^{0}\}$ are the
Neumann eigenvalues of $-\Delta$ in increasing order including
multiplicities with respect to the unit square and the perturbed
square, then $\lambda_2^{\varepsilon}\rightarrow 0$ as $\varepsilon
\rightarrow 0$, but $\lambda_2^{0}>0$. This example shows that one
needs additional regularity in order to achieve convergence.
Furthermore, Arrieta, Hale, and Han $\cite{Arrieta}$ show that for
this type of domain, $\lambda_m^{\varepsilon}\rightarrow
\lambda_{m-1}^{0}$, as $\varepsilon \rightarrow 0$ for $m\geq 3$.
Another work of Arrieta \cite{Arrieta2} gives rates of convergence
for eigenvalues of the Neumann Laplacian on a dumbbell domain in
$\mathbb{R}^2$ when the tube is more general. Jimbo and Morita
\cite{Jimbo2} study the first $N$ eigenvalues of the Neumann
Laplacian in $N$ disjoint domains connected by thin tubes. They show
that the first $N$ eigenvalues approach zero and the $(N+1)$st
eigenvalue is uniformly bounded away from zero. If $D_1$ and $D_2$ are two disjoint domains, then for
$\{\sigma_k\}=\{\mu_l\}\cup \{\lambda_j\}$, where $\{\mu_l\}$ are
the Neumann eigenvalues of $-\Delta$ in $D=D_1\cup D_2$ and
$\{\lambda_j\}$ are the Dirichlet eigenvalues of $\frac{-d^2}{dx^2}$
in $(-1,1)$, Jimbo \cite{Jimbo1} gives a rate of convergence on the
difference $\sigma_k-\sigma_k^{\varepsilon}$. This work was
generalized to more classes of domains in a more recent work by
Jimbo and Kosugi \cite{Jimbo3}. Also, Brown, Hislop, and Martinez
$\cite{Brown1}$ show that if $\sigma_k\in \{\mu_l\} \backslash
\{\lambda_j\}$ then
$$|\sigma_k-\sigma_k^{\varepsilon}|\leq C\left[\log
\left(\frac{1}{\varepsilon}\right)\right]^{\frac{-1}{2}}~~~~n=2$$
$$|\sigma_k-\sigma_k^{\varepsilon}|\leq C\varepsilon^{\frac{n-2}{2}}~~~~n\geq 3.$$

Here, we aim to provide an outline of the proof. In section 2, we give several definitions and describe the family of
domains for which we can prove the convergence of eigenvalues. We
also describe the well-known construction of eigenvalues and state
our main result. In section 3, we give Theorem \ref{thm2} from Giaquinta
and Modica \cite{Giaquinta}, \cite{Modica} which uses a technique
introduced by Gehring \cite{Gehring}. We also prove a Caccioppoli
type estimate for eigenfunctions in Theorem \ref{thm5} and use this along with Theorem \ref{thm2} to obtain a reverse
H\"{o}lder inequality given in Theorem \ref{revthm1}. This gives $L^p$-integrability for the gradient of the eigenfunctions for $p>2$. In section 4, we are able to bound these $L^p$ norms by a constant in Proposition \ref{prop3}. The proof uses the reverse H\"{o}lder
inequality as the key ingredient. This estimate is then used to prove Lemma \ref{lemma2} and Proposition \ref{prop5}, which are used to satisfy the first part of a well-known theorem from Ann\'{e} \cite{Anne2} given in Lemma \ref{lemma1}. The second part of Lemma \ref{lemma1} follows from the first part along with the above estimates, thus giving Corollary \ref{cor1}. The main result follows from this corollary. As
a by-product of our research, we give a simple proof of Shi and
Wright's \cite{Shi} $L^p$-estimates for the gradient of the Lam\'{e}
system as well as other elliptic systems. Many of the results first appeared in the author's Ph.D. dissertation \cite{Taylor}. \vspace{2cm}

\section{Preliminaries and Main Result}

\indent

We give conditions on a family of domains $\Omega_{\varepsilon}$
that allow us to prove the convergence of eigenvalues. We let
$\Omega$ and $\widetilde{\Omega}$ in $\mathbb{R}^n$ be two
non-empty, open, disjoint, and bounded sets. We let
$\varepsilon_1>0$ (which will be chosen small later), and then let
$\{T_{\varepsilon}\}_{0<\varepsilon\leq\varepsilon_1}$ be a family
of open sets such that $$T_{\widetilde{\varepsilon}}\subset
T_{\varepsilon}~~~~\textrm{if}~~\widetilde{\varepsilon}\leq
\varepsilon$$ and if $|T_{\varepsilon}|$ denotes the Lebesgue
measure of $T_{\varepsilon}$, then
\begin{equation} \label{eqn12} |T_{\varepsilon}|\leq
C\varepsilon^{d}
\end{equation} where $C$ and $0<d\leq n$ are independent of $\varepsilon$. Fix two points $p_1$ and $p_2$ on $\partial \Omega$ and
$\partial \widetilde{\Omega}$, respectively. For each $\varepsilon$, let $B_{\varepsilon}$ and
$\widetilde{B}_{\varepsilon}$ be two balls of radius
$\varepsilon$ in $\mathbb{R}^n$ centered at $p_1$ and $p_2$, respectively. The connections from $T_{\varepsilon}$ to $\Omega$ and
$\widetilde{\Omega}$ will be contained in $B_{\varepsilon}$ and
$\widetilde{B}_{\varepsilon}$, so that $T_{\varepsilon}\cap
\Omega=\emptyset$ and $\overline{T_{\varepsilon}}\cap
\overline{\Omega}\subset B_{\frac{\varepsilon}{2}}$ where
$B_{\frac{\varepsilon}{2}}$ is the concentric ball to
$B_{\varepsilon}$ of radius $\frac{\varepsilon}{2}$. Also, suppose a
similar condition for $\widetilde{\Omega}$ and
$\widetilde{B}_{\varepsilon}$. Then for any $\varepsilon$, define
$\Omega_{\varepsilon}$ to be the set $\Omega \cup
 \widetilde{\Omega} \cup T_{\varepsilon}$, which we assume to be
 open and connected,
and $\Omega_0=\Omega \cup \widetilde{\Omega}$. So, if our family is the family of dumbbell domains, you may think of
$T_{\varepsilon}$ as a ``tube'' connecting each of the two domains.
We now have the family of domains
$\{\Omega_{\varepsilon}\}_{0\leq\varepsilon\leq\varepsilon_1}$.

 \indent Next, we give a condition on the boundary of
$\Omega_{\varepsilon}$. If $B_r$ is any ball of radius $r$
satisfying $B_{r}\cap \Omega_{\varepsilon}^c \neq \emptyset$, then
\begin{equation} \label{eqn11} |B_{2r}\cap \Omega_{\varepsilon}^c| \geq C_0r^n\end{equation} where $C_0$ is a constant independent of $r$ and $\varepsilon$. This eliminates domains with ``cracks'' and ``in-cusps,'' and will be used to help show the Caccioppoli inequality in Theorem \ref{thm5} for the case when we are close to the boundary.

Throughout this paper we use the convention of summing over repeated
indices, where $i$ and $j$ will run from 1 to $n$ and $\alpha$,
$\beta$, and $\gamma$ will run from 1 to $m$. We let $a^{\alpha
\beta}_{ij}(x)$ be bounded, measurable, real-valued functions on
$\mathbb{R}^n$ which satisfy the symmetry condition
$$a^{\alpha \beta}_{ij}(x)=a^{\beta
\alpha}_{ji}(x),~~~i,j=1,2,...,n,~~~\alpha,\beta=1,2,...,m.$$ We let
$L^2(\Omega_{\varepsilon})$ denote the space of square integrable
functions taking values in $\mathbb{R}^m$ and
$H^1_0(\Omega_{\varepsilon})$ denotes the Sobolev space of vector-valued functions having one derivative in
$L^2(\Omega_{\varepsilon})$ and which vanish on the boundary. We use
$u^{\alpha}_j$ to denote the partial derivative $\frac{\partial
u^{\alpha}}{\partial x_j}$.

Let $\eta_{\varepsilon} \in C^{\infty}_c(\mathbb{R}^n)$ be a cutoff
function so that $\eta_{\varepsilon}=0$ in $T_{\varepsilon}$,
$\eta_{\varepsilon}=1$ in $\Omega_0 \backslash (B_{\varepsilon}\cup
\widetilde{B_{\varepsilon}})$, $|\nabla \eta_{\varepsilon}|\leq
\frac{C_n}{\varepsilon}$, and $0\leq \eta_{\varepsilon}\leq1$, where
$C_n$ only depends on $n$. We emphasize that $B_{\varepsilon}$,
$\widetilde{B_{\varepsilon}}$, and $\eta_{\varepsilon}$ depend on
the parameter $\varepsilon$. With these assumptions and definitions,
we have that for any $u\in H_0^1(\Omega_{\varepsilon})$,
$\eta_{\varepsilon} u$ will be in $H_0^1(\Omega_{0})$.

We now introduce the notion of an eigenvalue and corresponding
eigenvector. We say that the number $\sigma$ is a Dirichlet
eigenvalue of $L$ with Dirichlet eigenfunction $ u \in
H^1_0(\Omega)$, if $u\neq 0$ and
\begin{equation}\label{meqn40} \int_{\Omega}a^{\alpha
\beta}_{ij}(x)u^{\alpha}_i(x)\phi^{\beta}_j(x)~dx=\sigma
\int_{\Omega}u^{\gamma}(x)\phi^{\gamma}(x)~dx, \hspace{.5in} for
~any ~\phi \in H^1_0(\Omega). \end{equation}

\noindent We say that $L$ satisfies the {\it Legendre-Hadamard
condition} if there exists $\theta >0$ so that
\begin{equation} \label{eqn1aa} a^{\alpha
\beta}_{ij}(x)\xi_{\alpha}\xi_{\beta}\psi_i\psi_j\geq \theta
|\xi|^2|\psi|^2,~~~\xi\in\mathbb{R}^{m},~~~\psi\in\mathbb{R}^{n},~~~a.e.~x\in
\Omega_{\varepsilon}.
\end{equation}
If we define the norm on matrices $A=A_{ij}\in \mathbb{R}^{m\times
n}$ as $|A|^2=\displaystyle \sum_{i=1}^m\sum_{j=1}^n|A_{ij}|^2$, and
$L$ satisfies the {\it Legendre-Hadamard condition} with continuous
coefficients in $\overline{\Omega}$, then it is well-known that for
any $u\in H^1_0(\Omega)$, we have G{\aa}rding's inequality
\cite[p.~347]{Treves}
\begin{equation} \label{eqn1111} C_1\int_{\Omega}|\nabla u|^2~dx\leq
\int_{\Omega}a^{\alpha
\beta}_{ij}(x)u^{\alpha}_i(x)u^{\beta}_j(x)~dx+C_2\int_{\Omega}|u|^2~dx.
\end{equation}
\noindent $L$ is said to satisfy a {\it strong Legendre
condition} or a {\it strong ellipticity condition} if there exists
$\theta >0$ so that
\begin{equation} \label{eqn1} a^{\alpha
\beta}_{ij}(x)\xi^{\alpha}_i\xi^{\beta}_j\geq \theta
|\xi|^2,~~\xi\in\mathbb{R}^{m\times n},~~~a.e.~x\in
\Omega_{\varepsilon}.
\end{equation}

We introduce the Lam\'{e} system as $Lu=-\textrm{div}\zeta(u)$,
where $\zeta(u)$ denotes the stress tensor defined by
\begin{equation} \label{eqn2} \zeta^{\beta}_j(u):=a^{\alpha
\beta}_{ij}u^{\alpha}_i\end{equation} which is defined in terms of
the Lam\'{e} moduli $\upsilon(x)$ and $\mu(x)$ by
\begin{equation} \label{eqn3}a^{\alpha \beta}_{ij}(x)=\upsilon(x)
\delta_{i\alpha} \delta_{j \beta}+\mu(x)\delta_{ij} \delta_{\alpha
\beta}+\mu(x)\delta_{i\beta} \delta_{j \alpha},\end{equation} where $\upsilon(x)$ and $\mu(x)$ are both assumed to be bounded and measurable. Also,
define the strain tensor $\kappa(u)$ as
\begin{equation}\label{eqn4}
\kappa_{ij}(u):=\frac{1}{2}\left(u^i_j+u^j_i\right).\end{equation}
Note that for the Lam\'{e} system, $m=n$ and the Lam\'{e} parameters
$\upsilon(x)$ and $\mu(x)$ given in ($\ref{eqn3}$) satisfy
the conditions
\begin{equation} \label{eqn5}\upsilon(x)\geq0~~~~~~~~\mu(x)\geq
\delta >0.\end{equation}
With these assumptions, the Lam\'{e} system satisfies the ellipticity condition
\begin{equation} \label{eqn6}a^{\alpha
\beta}_{ij}u_i^{\alpha}u_j^{\beta}\geq
\tau\left|\kappa(u)\right|^2,~~~u\in H^1_0(\Omega_{\varepsilon})
\end{equation}
where $\tau=2\delta$. With Korn's 1st Inequality, it is easy to see that
for the Lam\'{e} system, we have
$$\frac{\tau}{2} \int_{\Omega_{\varepsilon}}|\nabla u|^2~dy \leq \int_{\Omega_{\varepsilon}}a^{\alpha
\beta}_{ij}u_i^{\alpha}u_j^{\beta}~dy,~~~~u\in
H^1_0(\Omega_{\varepsilon}).$$ Thus, if u satisfies either the
ellipticity condition (\ref{eqn1}), (\ref{eqn6}), or (\ref{eqn1aa})
with continuous coefficients in $\overline{\Omega}$, then we have
G{\aa}rding's inequality (\ref{eqn1111}).

The well-known construction of eigenvalues and eigenfunctions for scalar functions (which is the same for vector-valued functions) is taken from
Gilbarg and Trudinger \cite[p.~212]{Gilbarg}. If
we define the bilinear form on $H^1_0(\Omega_{\varepsilon})\times
H^1_0(\Omega_{\varepsilon})$ as
\begin{equation} \label{eqn7}
B_{\varepsilon}(u,v):=\int_{\Omega_{\varepsilon}}a^{\alpha
\beta}_{ij}u_i^{\alpha}v_j^{\beta}~dx \end{equation} and define the
Rayleigh quotient $R_{\varepsilon}$ as
\begin{equation} \label{eqn8}
R_{\varepsilon}(u):=\frac{B_{\varepsilon}(u,u)}{\|u\|_{L^2(\Omega_{\varepsilon})}^2}
 \end{equation} for $u\neq 0$, then we can construct an increasing
sequence of eigenvalues, listed according to multiplicity,
$\{\sigma_k\}_{k=1}^{\infty}$ such that for each corresponding
eigenfunction $u_k \in H^1_0(\Omega_{\varepsilon})$, we have
\begin{equation} \label{eqn9} \min_{w\in
\{u_1,...,u_{k-1}\}^{\perp}}R_{\varepsilon}(w)=R_{\varepsilon}(u_k)=\sigma_k\end{equation}
and
\begin{equation} \label{eqn10}
\|u_k\|_{L^2(\Omega_{\varepsilon})}=1 \end{equation}for any $k$.
\noindent Furthermore, each eigenspace is finite-dimensional and the
constructed set of eigenfunctions forms an orthonormal basis in
$L^2(\Omega_{\varepsilon})$.

 We now state the main result.
\begin{theorem} \label{thm1} Let $$(Lu)^{\beta}=-\frac{\partial}{\partial x_j}\left(a^{\alpha \beta}_{ij}\frac{\partial u^{\alpha}}{\partial x_i}\right)\hspace{.5in} \beta=1,...,m$$ satisfy one of the following:
\begin{enumerate}
\item $L$ has uniformly bounded coefficients and satisfies either the
ellipticity condition (\ref{eqn1}) or the ellipticity condition
(\ref{eqn6}).
\item $L$ has continuous coefficients and satisfies the
ellipticity condition (\ref{eqn1aa}). \end{enumerate} Also assume
$\{\sigma_k^0\}_{k=1}^{\infty}$ and
$\{\sigma_k^{\varepsilon}\}_{k=1}^{\infty}$ are the Dirichlet
eigenvalues of $L$ with respect to $\Omega_0$ and
$\Omega_{\varepsilon}$ in increasing order numbered according to
multiplicity. Then for each $J\in \mathbb{N}$, we have the following
estimate: $$|\sigma_J^{\varepsilon}-\sigma_J^0|\leq
C\varepsilon^{a}$$ for $0<\varepsilon\leq \varepsilon_0(J)$, where
$\varepsilon_0(J)$ depends on the multiplicity of $\sigma_J^0$.
Moreover, the rate $a>0$ is independent of any eigenvalue and $C$
only depends on the eigenvalue $\sigma_J^0$ and the distance from $\sigma_J^0$ to
nearby eigenvalues.

\end{theorem}

\section{A Reverse H\"{o}lder Inequality}
\indent

 If $\displaystyle-\!\!\!\!\!\!\int_{E}|f(y)|~dy$ is defined
to be the average of $f$ on $E$, then recall that the maximal
function is defined for $f\in L^1_{loc}(\mathbb{R}^n)$ to be
$$M(f)(x):=\sup_{r>0}-\!\!\!\!\!\!\int_{B_r(x)}|f(y)|~dy$$ where
$B_r(x)$ is a ball of radius $r$ centered at $x$. Also, define
$M_R(f)(x)$ to be
$$M_R(f)(x):=\sup_{R>r>0}-\!\!\!\!\!\!\int_{B_r(x)}|f(y)|~dy.$$
We will need the following theorem from Giaquinta
\cite[p.~122]{Giaquinta}, which uses the technique introduced by
Gehring \cite{Gehring},  and refined by Giaquinta and Modica
\cite{Modica}.
\begin{theorem} \label{thm2} Let $r>q>1$, and $Q_R$ be a cube in
$\mathbb{R}^n$ with sidelength $R$ centered at 0. Also, define
$d(x)=dist(x,\partial Q_R)$. If $f$ and $g$ are measurable functions such that $f\in L^r(Q_R)$, $g\in L^q(Q_R)$, $f=g=0$
outside $Q_R$, and with the added condition that
$$M_{\frac{d(x)}{2}}(|g|^q)(x)\leq bM^q(g)(x)+M(|f|^q)+aM(|g|^q)(x)$$
for almost every $x$ in $Q_R$ where $b\geq0$ and $0\leq a<1$, then $g
\in L^p(Q_{\frac{R}{2}})$, for $p \in[q,q+\epsilon)$ and
\begin{equation} \label{meqn2} \left(-\!\!\!\!\!\!\int_{Q_{R/2}}|g|^p(y)~dy\right)^{\frac{1}{p}}\leq
C\left[\left(-\!\!\!\!\!\!\int_{Q_{R}}|g|^q(y)~dy\right)^{\frac{1}{q}}+\left(-\!\!\!\!\!\!\int_{Q_{R}}|f|^p(y)~dy\right)^{\frac{1}{p}}\right]\end{equation}
where $\epsilon$ and $C$ depend on $b,q,n,a$ and $r$.
\end{theorem}
The conclusion of this theorem is known as a reverse H\"{o}lder
inequality. To show that the gradient of eigenfunctions satisfy this
inequality, we will need to prove a Caccioppoli inequality. However,
to show this Caccioppoli inequality, we first need the following two
well-known inequalities taken from Hebey \cite[p.~44]{Hebey} and
Oleinik \cite[p.~27]{Oleinik}:
\begin{theorem} \label{thm3} {\bf Sobolev-Poincar\'{e} Inequality} Let $1\leq p<n$
and $\frac{1}{q}=\frac{1}{p}-\frac{1}{n}$. Also, let $B_r$ be any
ball of radius $r$ with $u\in W^{1,p}(B_r)$. Then, for $S$ contained
in $B_r$ with $|S|\geq c_0r^n$,
\begin{equation} \label{meqn3} \int_{B_r}|u(x)-u_{S}|^q~dx \leq C\left(\int_{B_r}|\nabla
u|^p(x)~dx\right)^{\frac{q}{p}}\end{equation} where
$u_S=\frac{1}{|S|}\int_Su~dy$, for some constant $C(n,p,c_0)$,
independent of $u$.
\end{theorem}
\begin{theorem} \label{thm4} {\bf Korn's Inequality on a Ball} If $u\in
H^1(B_r)$ then \begin{equation} \label{meqn4} \|\nabla
u\|^2_{L^2(B_r)}\leq
C\left(\|\kappa(u)\|^2_{L^2(B_r)}+\frac{1}{r^2}\|
u\|^2_{L^2(B_r)}\right)\end{equation} where $C$ only depends on
$n$.\end{theorem}

We now state and prove a Caccioppoli inequality for eigenfunctions:
\begin{theorem} \label{thm5} Let $u$ be an eigenfunction with eigenvalue $\sigma$
associated to the operator $L$ satisfying either (\ref{eqn1}) or
(\ref{eqn6}) with uniformly bounded coefficients or associated to
(\ref{eqn1aa}) with continuous coefficients. Extending $u$ to be 0
outside $\Omega_{\varepsilon}$, there exists $r_0>0$ so that if
$r_0\geq r>0$, $x\in \mathbb{R}^n$, we have

\begin{align} -\!\!\!\!\!\!\int_{B_r}|\nabla u|^2~dy\leq &
C_1\left(-\!\!\!\!\!\!\int_{B_{2r}}|\nabla
u|^{\frac{2n}{n+2}}~dy\right)^{\frac{n+2}{n}} \nonumber \\
& \label{meqn5}+C_2|\sigma|-\!\!\!\!\!\!\int_{B_{2r}}|
u|^2~dy+C_3-\!\!\!\!\!\!\int_{B_{2r}}|\nabla u|^2~dy
\end{align}
 where $B_r$ is a ball with radius $r$ centered at
$x$, $C_3<1$, and $C_l>0$ only depends on
$M=\max_{i,j,\alpha,\beta}\|a^{\alpha
\beta}_{ij}\|_{L^{\infty}(\Omega_{\varepsilon})}$, $n$, $m$,
$\theta$, $\tau$, and $C_0$. Furthermore, if $L$ satisfies either
(\ref{eqn1}) or (\ref{eqn6}) with uniformly bounded coefficients,
then the inequality holds for any $r>0$.\end{theorem}
\begin{proof} First, choose a ball $B_r$ and define a cutoff function $\nu \in C^{\infty}_c(\mathbb{R}^n)$ to be so that $\nu=1$ in
$B_r$, $\nu=0$ outside $B_{2r}$, $|\nabla \nu|\leq \frac{C_n}{r}$,
and $0\leq \nu \leq1$, where $C_n$ only depends on $n$. Below, we
will find an appropriate constant vector $\rho \in \mathbb{R}^m$, so
that $\nu^2(u-\rho)\in H^1_0(\Omega_{\varepsilon})$. By the weak
formulation (\ref{meqn40}), we have
$$\int_{\Omega_{\varepsilon}}a^{\alpha
\beta}_{ij}u^{\alpha}_i[\nu^2(u-\rho)]^{\beta}_j~dy=\sigma
\int_{\Omega_{\varepsilon}}u^{\gamma}[\nu^2(u-\rho)]^{\gamma}~dy.$$
Then, performing the differentiations, we get
\begin{equation}\label{eqn1bb} \int_{\Omega_{\varepsilon}}a^{\alpha
\beta}_{ij}u^{\alpha}_i[2\nu
\nu_j(u-\rho)^{\beta}+\nu^2u^{\beta}_j]~dy=\sigma
\int_{\Omega_{\varepsilon}}u^{\gamma}\nu^2(u-\rho)^{\gamma}~dy.\end{equation}
From this point, the argument depends on the ellipticity condition. We have 3 cases. \vspace{.25in}
\newline {\it case 1: $L$ satisfies the strong ellipticity condition (\ref{eqn1}).}\vspace{.25in} \newline \indent
Using (\ref{eqn1}) and properties of $\nu$, we obtain the inequality
$$\int_{B_{2r}}\nu^2 a^{\alpha
\beta}_{ij}u^{\alpha}_iu^{\beta}_j~dy\leq
\int_{B_{2r}}2M\frac{C_n}{r}\nu|\nabla
u||u-\rho|~dy+\int_{B_{2r}}|\sigma||u||u-\rho|~dy$$ which, for any
constant $\omega>0$, then leads to
\begin{align}  \int_{B_{2r}}\nu^2 a^{\alpha
\beta}_{ij}u^{\alpha}_iu^{\beta}_j~dy &\leq
\int_{B_{2r}}\frac{\omega \nu^2|\nabla u|^2}{2}~dy+\frac{C}{\omega
r^2}\int_{B_{2r}}|u-\rho|^2~dy \nonumber \\
&\label{meqn50}\hspace{.25in}+C|\sigma|\int_{B_{2r}}|u|^2~dy\end{align}
where $C$ depends on $M$ and $C_n$. Then choosing $\omega=\theta$ in
$(\ref{meqn50})$ gives
$$\frac{\theta}{2}\int_{B_{2r}}\nu^2|\nabla u|^2~dy \leq \frac{C}{\theta r^2}\int_{B_{2r}}|u-\rho|^2~dy+C|\sigma|\int_{B_{2r}}|u|^2~dy.$$
Then, multiplying both sides by $\frac{2}{\theta}$ and using that
$\nu=1$ on $B_r$ gives
\begin{equation} \label{eqnaaa1}\int_{B_{r}}|\nabla u|^2~dy
\leq  \frac{2C}{\theta^2
r^2}\int_{B_{2r}}|u-\rho|^2~dy+\frac{2C|\sigma|}{\theta}\int_{B_{2r}}|u|^2~dy.\end{equation}
Now, for the term $\displaystyle \int_{B_{2r}}|u-\rho|^2~dy$, we
must consider two subcases.\vspace{.25in}

\noindent {\it subcase A} \vspace{.25in} \newline \indent If
$B_{2r}\subset \Omega_{\varepsilon}$, then let $\displaystyle
\rho^{\alpha}=-\!\!\!\!\!\!\int_{B_{2r}}u^{\alpha}~dy$. Our
condition on the support of $\nu$ implies $\nu^2(u-\rho)\in
H^1_0(\Omega_{\varepsilon})$. So, setting $q=2$ and $S=B_{2r}$ in
the Sobolev-Poincar\'{e} Inequality (\ref{meqn3}), we obtain
$$\int_{B_{2r}}|u-\rho|^2~dy\leq C\left(\int_{B_{2r}}|\nabla
u|^{\frac{2n}{n+2}}~dy\right)^{\frac{n+2}{n}}.$$ Using this estimate
with (\ref{eqnaaa1}) gives
$$\int_{B_{r}}|\nabla u|^2~dy
\leq \frac{C}{ r^2}\left(\int_{B_{2r}}|\nabla
u|^{\frac{2n}{n+2}}~dy\right)^{\frac{n+2}{n}}+C|\sigma|\int_{B_{2r}}|u|^2~dy.$$
Now, dividing through by $r^n$ gives the desired result with
$C_3=0$.\vspace{.25in}

\noindent {\it subcase B} \vspace{.25in} \newline \indent If
$B_{2r}\cap \Omega_{\varepsilon}^c\neq \emptyset$, then set
$\rho=0$, which, again, guarantees that $\nu^2(u-\rho)\in
H^1_0(\Omega_{\varepsilon})$. So setting $q=2$ and $S=B_{4r}\cap
\Omega_{\varepsilon}$ in the Sobolev-Poincar\'{e} Inequality
(\ref{meqn3}), we have by our assumption on $\Omega_{\varepsilon}^c$
(\ref{eqn11}) that
$$\int_{B_{4r}}|u-\rho|^2~dy\leq C\left(\int_{B_{4r}}|\nabla
u|^{\frac{2n}{n+2}}~dy\right)^{\frac{n+2}{n}}.$$ From
(\ref{eqnaaa1}), we obtain
$$\int_{B_{r}}|\nabla u|^2~dy
\leq \frac{C}{ r^2}\left(\int_{B_{4r}}|\nabla
u|^{\frac{2n}{n+2}}~dy\right)^{\frac{n+2}{n}}+C|\sigma|\int_{B_{4r}}|u|^2~dy.$$
A simple covering argument gives the estimate with $B_{4r}$ replaced
with $B_{2r}$.\vspace{.25in}
\newline \noindent {\it case 2: $L$ satisfies
the ellipticity condition (\ref{eqn6}).} \vspace{.25in} \newline
\indent From (\ref{eqn6}) and $(\ref{meqn50})$, we have
$$\int_{B_{r}} \tau |\kappa(u)|^2~dy\leq \int_{B_{2r}}\frac{\omega
\nu^2|\nabla u|^2}{2}~dy+\frac{C}{\omega
r^2}\int_{B_{2r}}|u-\rho|^2~dy+C|\sigma|\int_{B_{2r}}|u|^2~dy.$$
Also, by Korn's inequality (\ref{meqn4}), we have
$$\frac{\tau}{C}\int_{B_{r}}|\nabla
u|^2~dy-\frac{\tau}{r^2}\int_{B_{r}}|u-\rho|^2~dy\leq \int_{B_{r}}
\tau |\kappa(u)|^2~dy.$$ This implies 
\begin{align*}\int_{B_{r}}|\nabla
u|^2~dy\leq & \frac{C\omega}{2\tau}\int_{B_{2r}}|\nabla
u|^2~dy+C\left(\frac{1}{\omega \tau
r^2}+\frac{1}{r^2}\right)\int_{B_{2r}}|u-\rho|^2~dy\\&+\frac{C|\sigma|}{\tau}\int_{B_{2r}}|u|^2~dy.
\end{align*}
This again leads to two subcases. We must choose $\rho$
appropriately and use the Sobolev-Poincar\'{e} inequality
(\ref{meqn3}) as in case 1. Then, by taking $\omega$ sufficiently
small, we obtain the desired result. \vspace{.25in}
\newline \noindent {\it case 3: $L$ satisfies the Legendre-Hadamard condition
$(\ref{eqn1aa})$ with continuous coefficients in
$\overline{\Omega}_{\varepsilon}$.}\vspace{.25in}

We note that it suffices to study when $u\in
C^{\infty}_c(\Omega_{\varepsilon})$ and first consider when the
coefficients are constant. We rewrite the left side of
$(\ref{eqn1bb})$ as
\begin{align*}&\int_{\Omega_{\varepsilon}}a^{\alpha \beta}_{ij}((u-\rho)^{\alpha}\nu)_i((u-\rho)^{\beta}\nu)_j~dy\\&+\int_{\Omega_{\varepsilon}}a^{\alpha \beta}_{ij}[\nu
\nu_ju^{\alpha}_i(u-\rho)^{\beta}-\nu_i\nu
(u-\rho)^{\alpha}u^{\beta}_j-\nu_i\nu_j(u-\rho)^{\alpha}(u-\rho)^{\beta}]~dy.\end{align*}
This implies 
\begin{align*}&\int_{B_{2r}}a^{\alpha
\beta}_{ij}((u-\rho)^{\alpha}\nu)_i((u-\rho)^{\beta}\nu)_j~dy\\ & \leq 
C\int_{B_{2r}}|\nabla
\nu||\nabla((u-\rho)\nu)||u-\rho|+|u-\rho|^2|\nabla
\nu|^2+|\sigma||u||u-\rho|~dy.\end{align*} We note that we may use the Fourier
transform to get a lower bound for the left side to achieve the
estimate

$$
\int_{B_{2r}}|\nabla((u-\rho)\nu)|^2~dy\leq
\frac{C}{r^2}\int_{B_{2r}}|u-\rho|^2~dy+C|\sigma|\int_{B_{2r}}|u|^2~dy.
$$ This implies the estimate
\begin{equation}
\label{eqnaaa2} \int_{B_{r}}|\nabla
u|^2~dy\leq
\frac{C}{r^2}\int_{B_{2r}}|u-\rho|^2~dy+C|\sigma|\int_{B_{2r}}|u|^2~dy.
\end{equation}
So, again, if we employ the Sobolev-Poincar\'{e} inequality
(\ref{meqn3}), we get the desired result in the case of constant
coefficients. If the coefficients are continuous and non-constant,
then we freeze the coefficients at $x$. That is, from the weak
formulation (\ref{meqn40}), we have
\begin{align}\int_{\Omega_{\varepsilon}}a^{\alpha
\beta}_{ij}(x)u^{\alpha}_i((u-\rho)\nu^2)^{\beta}_j~dy&+\int_{\Omega_{\varepsilon}}(a^{\alpha
\beta}_{ij}-a^{\alpha
\beta}_{ij}(x))u^{\alpha}_i((u-\rho)\nu^2)^{\beta}_j~dy \nonumber\\
&\label{eqn1dd}=\sigma\int_{\Omega_{\varepsilon}}u^{\gamma}((u-\rho)\nu^2)^{\gamma}~dy.
\end{align}
So, if we define the modulus of continuity to be
$$M(x_0,R)=\max_{\substack{ y\in \overline{B_{R}(x_0)}\\i,j,\alpha, \beta}}|a^{\alpha
\beta}_{ij}(y)-a^{\alpha \beta}_{ij}(x_0)|$$ then we have that
\begin{align*}\lefteqn{ \int_{B_{2r}}(a^{\alpha \beta}_{ij}-a^{\alpha
\beta}_{ij}(x))u^{\alpha}_i((u-\rho)\nu^2)^{\beta}_j~dy}\\
&\hspace{.5in}\leq M(x,2r)\int_{B_{2r}}\nu^2|\nabla
u|^2~dy+2M(x,2r)\int_{B_{2r}}\nu |\nabla\nu||\nabla u||u-\rho|~dy \\
&\hspace{.5in}\leq C(M(x,2r)+M(x,2r)^2)\int_{B_{2r}}|\nabla
u|^2~dy+\frac{C}{r^2}\int_{B_{2r}}|u-\rho|^2~dy.
\end{align*}
 Also, by the
uniform continuity of the coefficients on
$\overline{\Omega}_{\varepsilon}$, for any $c<1$, there exists $r_0$
depending on $c$, so that if $C(x_0,R)=C(M(x_0,2R)+M(x_0,2R)^2)$ and
$r\leq r_0$, then
$$C(x_0,r)\leq c$$ for all $x_0\in \overline{\Omega}_{\varepsilon}$.
So, now moving the second term on the left side of $(\ref{eqn1dd})$
to the right and using the constant coefficient case
(\ref{eqnaaa2}), we obtain that for any $c<1$, there exists $r_0$ so
that if $r\leq r_0$,
$$\int_{B_{r}}|\nabla
u|^2~dy\leq
\frac{C}{r^2}\int_{B_{2r}}|u-\rho|^2~dy+C|\sigma|\int_{B_{2r}}|u|^2~dy+c\int_{B_{2r}}|\nabla
u|^2~dy.$$ We again choose $\rho$ appropriately and apply the
Sobolev-Poincar\'{e} inequality (\ref{meqn3}) to get the desired
result.
\end{proof}

As stated earlier, our proof of Theorem \ref{thm1} relies on the
gradient of an eigenfunction satisfying the reverse H\"{o}lder
inequality, as in our next theorem.

\begin{theorem} \label{revthm1} There exists $\epsilon_1>0$ so that if $u$ is an eigenfunction with eigenvalue $\sigma$, then
\begin{equation} \label{eqn13}
-\!\!\!\!\!\!\int_{\Omega_{\varepsilon}}|\nabla
u|^{\widetilde{p}}~dy\leq
C\left[\left(-\!\!\!\!\!\!\int_{\Omega_{\varepsilon}}|\nabla
u|^2~dy\right)^{\frac{\widetilde{p}}{2}}+|\sigma|^{\frac{\widetilde{p}}{2}}-\!\!\!\!\!\!\int_{\Omega_{\varepsilon}}|u|^{\widetilde{p}}~dy\right]\end{equation}
where $2\leq \widetilde{p}< 2+\epsilon_1$, and $\epsilon_1$ and $C$
are independent of $\varepsilon$ and any eigenvalue.
\end{theorem}

\begin{proof}Now if $u$ is an eigenfunction with eigenvalue $\sigma$, we
have $u\in H^1_0(\Omega_{\varepsilon})$, and thus we may employ the
Sobolev inequality to get that $|u|\in L^r(\Omega_{\varepsilon})$
for some $r>2$. If $L$ satisfies either (\ref{eqn1}) or (\ref{eqn6})
with uniformly bounded coefficients, then we may choose a cube
$Q_R$, centered at 0, with sidelength $R$ such that
$\Omega_{\varepsilon}\subset Q_{\frac{R}{2}}$, uniformly in
$\varepsilon$, and set $g=|\nabla u|^{\frac{2n}{n+2}}$,
$f=(C_3|\sigma|)^{\frac{n}{n+2}}|u|^{\frac{2n}{n+2}}$,
$q=\frac{n+2}{n}$, and $u=0$ outside $\Omega_{\varepsilon}$, we may
conclude by (\ref{meqn5}) and (\ref{meqn2}) that
$$\left(-\!\!\!\!\!\!\int_{\Omega_{\varepsilon}}|\nabla
u|^{\frac{2np}{n+2}}~dy\right)^{\frac{1}{p}}\leq
C\left[\left(-\!\!\!\!\!\!\int_{\Omega_{\varepsilon}}|\nabla
u|^2~dy\right)^{\frac{n}{n+2}}+|\sigma|^{\frac{n}{n+2}}\left(-\!\!\!\!\!\!\int_{\Omega_{\varepsilon}}|u|^{\frac{2np}{n+2}}~dy\right)^{\frac{1}{p}}\right]$$
where $\frac{n+2}{n}\leq p \leq \frac{n+2}{n}+\epsilon$, which, from
Theorem $\ref{thm5}$ is independent of $\varepsilon$ and any
eigenvalue. So, setting $\widetilde{p}=\frac{2np}{n+2}$, we have the
result. If $L$ satisfies (\ref{eqn1aa}) with continuous
coefficients, then since we only have Theorem $\ref{thm5}$ true for
small $r$, we must cover $\Omega_{\varepsilon}$ with a fixed number
of cubes and apply (\ref{meqn2}) to each cube to obtain the
result.\end{proof}

\section{Stability of Eigenvalues}

\indent

From this point, let $\sigma^{\varepsilon}_k$ be the $kth$
eigenvalue with respect to $\Omega_{\varepsilon}$, and
$\phi^{\varepsilon}_k$ be its corresponding eigenfunction with
$\phi^{\varepsilon}_k=0$ outside $\Omega_{\varepsilon}$ for
$\varepsilon\geq 0$. We also {\it fix} an eigenvalue $\sigma^{0}_J$
with multiplicity $m_J$ where $\sigma^{0}_{J-1}<\sigma^{0}_J$ if
$J\geq2$. We will consider the family $\{\sigma^{\varepsilon}_J\}$
as $\varepsilon>0$ tends to 0. We begin with the following
proposition taken from Ann\'{e} \cite[p.~2595-2596]{Anne2}.

\begin{lemma} \label{lemma1} Let $(q,\cal{D})$ be a closed non-negative quadratic
form with form domain $\cal{D}$ in the Hilbert space
$(\cal{H},\langle \cdot, \cdot \rangle)$. Define the associated norm
$\|f\|_1^2=\|f\|^2_{{\cal H}}+q(f)$, and the spectral projector
$\Pi_I$ for any interval $I=(\alpha,\beta)$ for which the boundary
does not meet the spectrum.
\begin{enumerate}
\item  Suppose $f\in \cal{D}$ and $\lambda \in I$ satisfy
$$|q(f,g)-\lambda \langle f,g \rangle|\leq \delta \|f\|\|g\|_1~~~~~g\in \cal{D}.$$
Then there exists a constant $C>0$, which depends on $I$, such that
if $a$ is less than the distance of $\alpha$ or $\beta$ to the
spectrum of $q$,
$$\|\Pi_I(f)-f\|_1=\|\Pi_{I^c}(f)\|_1\leq \frac{C\delta}{a}\|f\|.$$
\item Suppose the spectral space $E(I)$ has dimension $m$ and $f_1,...,f_m$ is an orthonormal family which satisfies
$$\|\Pi_{I^c}(f_j)\|_1\leq \delta~~~~~ j=1,...,m.$$ Also let $E$ be
the space spanned by the $f_j$'s. Then, $$\textrm{dist}(E(I),E)\leq
C\delta$$ where the distance is measured as the distance between the
two orthogonal projectors.
\end{enumerate}
\end{lemma}
\noindent This lemma will give us the results we need for the
convergence of eigenvalues. We will prove estimates on
eigenfunctions using the reverse H\"{o}lder inequality
(\ref{eqn13}), which will allow us to use this lemma. We start with
the following proposition which follows immediately from the
construction of eigenvalues.

\begin{prop} \label{prop2} We have for any $\varepsilon >0$, and any $k\in \mathbb{N}$, \begin{equation} \label{eqn14}\sigma^{\varepsilon}_k\leq \sigma^{0}_k. \end{equation} \end{prop}
\noindent This proposition gives us the easy half of the inequality
in our theorem. To prove the second half of the inequality, we will
need a few items.

\begin{prop} \label{prop3} For any $\varepsilon>0$, and $k\geq1$, if $\phi=\phi^{\varepsilon}_k$, then we have
\begin{equation} \label{eqn15} \int_{\Omega_{\varepsilon}}|\nabla \phi|^{\widetilde{p}}~dy\leq
C\end{equation} where ${\widetilde{p}}>2$ is from $(\ref{eqn13})$,
and $C$ depends on $|\Omega_0|$ and $n$, with order ${\cal
O}\left(|\sigma^{0}_k|^{\frac{2\widetilde{p}+n(\widetilde{p}-2)}{4}}\right)$
for $n\geq3$ or ${\cal
O}\left(|\sigma^{0}_k|^{\frac{q\widetilde{p}+2(\widetilde{p}-q)}{2q}}\right)$
for $n=2$ where $2-\xi<q<2$ for small $\xi$. Furthermore,
$\widetilde{p}$ and $C$ are independent of $\varepsilon$ and if
$n=2$, $C$ blows up as $q\rightarrow 2$.
\end{prop}
\begin{proof} Now, from $(\ref{eqn13})$, we have
\begin{equation}\label{eqnaaa5}\int_{\Omega_{\varepsilon}}|\nabla \phi|^{\widetilde{p}}~dy\leq
C\left[|\Omega_{\varepsilon}|^{\frac{2-\widetilde{p}}{2}}\left(\int_{\Omega_{\varepsilon}}|\nabla
\phi|^2~dy\right)^{\frac{\widetilde{p}}{2}}+|\sigma^{\varepsilon}_k|^{\frac{\widetilde{p}}{2}}\left(\int_{\Omega_{\varepsilon}}|\phi|^{\widetilde{p}}~dy\right)\right]\end{equation}
where $\widetilde{p}>2$ is from $(\ref{eqn13})$. Recall that by
G{\aa}rding's inequality (\ref{eqn1111}) and since $\phi$ is an eigenfunction, we have
\begin{align}
C_1\int_{\Omega_{\varepsilon}}|\nabla
\phi|^{2}~dy&\leq\int_{\Omega_{\varepsilon}}a_{ij}^{\alpha
\beta}\phi^{\alpha}_i\phi^{\beta}_j~dy+C_2\int_{\Omega_{\varepsilon}}|\phi|^{2}~dy\notag\\
&\leq
C(1+|\sigma^{\varepsilon}_k|)\int_{\Omega_{\varepsilon}}|\phi|^{2}~dy\notag\\
&\label{eqnaaa4}\leq C(1+|\sigma^{\varepsilon}_k|),
\end{align}
the last line owing to the normalization of the eigenfunctions.
Next, we will consider $n\geq3$ and estimate
$$\int_{\Omega_{\varepsilon}}|\phi|^{\widetilde{p}}~dy.$$

Using Sobolev's inequality and (\ref{eqnaaa4}), we have

\begin{align}
\left(\int_{\Omega_{\varepsilon}}|
\phi|^{\frac{2n}{n-2}}~dy\right)^{\frac{n-2}{2n}}&\leq
C\left(\int_{\Omega_{\varepsilon}}|\nabla
\phi|^{2}~dy\right)^{\frac{1}{2}}\notag\\
&\label{eqnaaa3}\leq C(1+|\sigma^{\varepsilon}_k|^{\frac{1}{2}}).
\end{align}
Also, by H\"{o}lder's inequality, we have
\begin{align*}
\left(\int_{\Omega_{\varepsilon}}|
\phi|^{\widetilde{p}}~dy\right)^{\frac{1}{\widetilde{p}}}&\leq\left(\int_{\Omega_{\varepsilon}}|
\phi|^{2}~dy\right)^{\frac{1-t}{2}}\left(\int_{\Omega_{\varepsilon}}|
\phi|^{\frac{2n}{n-2}}~dy\right)^{\frac{t(n-2)}{2n}}\\
\end{align*}
where $t$ satisfies
$$\frac{1}{\widetilde{p}}=\frac{1-t}{2}+\frac{t(n-2)}{2n}.$$
From this inequality and (\ref{eqnaaa3}), it follows that
\begin{align*}\left(\int_{\Omega_{\varepsilon}}|
\phi|^{\widetilde{p}}~dy\right)^{\frac{1}{\widetilde{p}}}&\leq
C\left(1+|\sigma^{\varepsilon}_k|^{\frac{t}{2}}\right)\\
&=C\left(1+|\sigma^{\varepsilon}_k|^{\frac{n(\widetilde{p}-2)}{4\widetilde{p}}}\right).\\
\end{align*}
 Now, using this inequality along with (\ref{eqnaaa5}), (\ref{eqnaaa4}), and $(\ref{eqn14})$,
 we obtain
\begin{align*}\int_{\Omega_{\varepsilon}}|\nabla \phi|^{\widetilde{p}}~dy&\leq
C\left[\left(1+|\sigma^{0}_k|\right)^{\frac{\widetilde{p}}{2}}+|\sigma^{0}_k|^{\frac{\widetilde{p}}{2}}\left(1+|\sigma^{0}_k|^{\frac{n(\widetilde{p}-2)}{4}}\right)\right]\\
&\leq
C\left[\left|\sigma^{0}_k\right|^{\frac{2\widetilde{p}+n(\widetilde{p}-2)}{4}}+\left|\sigma^{0}_k\right|^{\frac{\widetilde{p}}{2}}+1\right].
\end{align*}
This completes the proof for $n\geq3$.

If $n=2$, then from Sobolev's inequality, H\"{o}lder's inequality,
and (\ref{eqnaaa4}), we have
\begin{align*}\left(\int_{\Omega_{\varepsilon}}|\phi|^{q^*}~dy\right)^{\frac{1}{q^*}}&\leq \frac{C}{(2-q)^{\frac{1}{2}}}\left(\int_{\Omega_{\varepsilon}}|\nabla
\phi|^{q}~dy\right)^{\frac{1}{q}}\\
&\leq
\frac{C}{(2-q)^{\frac{1}{2}}}\left(\int_{\Omega_{\varepsilon}}|\nabla
\phi|^{2}~dy\right)^{\frac{1}{2}}|\Omega_{\varepsilon}|^{\frac{1}{q^*}}\\
&\leq
\frac{C}{(2-q)^{\frac{1}{2}}}\left(1+|\sigma^{\varepsilon}_k|^{\frac{1}{2}}\right)
\end{align*} where $q^*=\frac{2q}{q-2}$ is the Sobolev conjugate of
$q$. Then, again applying H\"{o}lder's inequality, we obtain
\begin{align*}\left(\int_{\Omega_{\varepsilon}}|\phi|^{\widetilde{p}}~dy\right)^{\frac{1}{\widetilde{p}}}&\leq  \frac{C}{(2-q)^{\frac{t}{2}}}\left(1+|\sigma^{\varepsilon}_k|^{\frac{t}{2}}\right)\\
&=
\frac{C}{(2-q)^{\frac{(\widetilde{p}-q)}{\widetilde{p}q}}}\left(1+|\sigma^{\varepsilon}_k|^{\frac{(\widetilde{p}-q)}{\widetilde{p}q}}\right).
\end{align*}
Now using (\ref{eqnaaa5}), (\ref{eqnaaa4}), and $(\ref{eqn14})$,
 we obtain
 \begin{align*}\int_{\Omega_{\varepsilon}}|\nabla \phi|^{\widetilde{p}}~dy&\leq
\frac{C}{(2-q)^{\frac{(\widetilde{p}-q)}{q}}}\left[\left(1+|\sigma^{0}_k|\right)^{\frac{\widetilde{p}}{2}}+|\sigma^{0}_k|^{\frac{\widetilde{p}}{2}}\left(1+|\sigma^{0}_k|^{\frac{(\widetilde{p}-q)}{q}}\right)\right]\\
&\leq
\frac{C}{(2-q)^{\frac{(\widetilde{p}-q)}{q}}}\left[\left|\sigma^{0}_k\right|^{\frac{q\widetilde{p}+2(\widetilde{p}-q)}{2q}}+\left|\sigma^{0}_k\right|^{\frac{\widetilde{p}}{2}}+1\right].
\end{align*}

\end{proof}

\begin{lemma} \label{lemma2} For the eigenfunction $\phi^{\varepsilon}_{k}$, $J\leq k\leq J+m_J-1$, and any $w\in H^1_0(\Omega_{0})$, we have the
following estimate:
\begin{equation}\left| \int_{\Omega_{0}}a_{ij}^{\alpha
\beta}(\eta_{\varepsilon}\phi^{\varepsilon}_{k})^{\alpha}_iw^{\beta}_j~dy-\sigma^{\varepsilon}_k
\int_{\Omega_{0}}(\eta_{\varepsilon}\phi^{\varepsilon}_{k})^{\alpha}w^{\alpha}~dy\right|\leq
C\varepsilon^{\frac{n(\widetilde{p}-2)}{2\widetilde{p}}}\|w\|_1
\end{equation}
where $\|w\|_1$ is from Lemma \ref{lemma1}  with $\displaystyle
q(f,g)=\int_{\Omega_{0}}a_{ij}^{\alpha \beta}f^{\alpha}_i
g^{\beta}_j~dy$, and $C$ only depends on $|\Omega_0|$, $n$,
$\sigma^{0}_J$, and is independent of $\varepsilon$.
\end{lemma}

\begin{proof}
First, recall that $w$ is extended to be 0 outside $\Omega_{0}$ and
$\phi^{\varepsilon}_{k}$ is extended to be 0 in $(B_{\varepsilon}\cup
\widetilde{B}_{\varepsilon})\cap \Omega_{\varepsilon}^c$. We
have
\begin{align*}
&\left|\int_{\Omega_{0}}a_{ij}^{\alpha
\beta}(\eta_{\varepsilon}\phi^{\varepsilon}_{k})^{\alpha}_iw^{\beta}_j~dy-\sigma^{\varepsilon}_k
\int_{\Omega_{0}}(\eta_{\varepsilon}\phi^{\varepsilon}_{k})^{\alpha}
w^{\alpha}~dy\right|\\& \hspace{3cm}\leq\left|\int_{\Omega_{0}}a_{ij}^{\alpha
\beta}[(\eta_{\varepsilon})_i(\phi^{\varepsilon}_{k})^{\alpha}w^{\beta}_j-(\eta_{\varepsilon})_j(\phi^{\varepsilon}_{k})^{\alpha}w^{\beta}]~dy\right|\\
&\hspace{3.5cm}+\left|\int_{\Omega_{\varepsilon}}a_{ij}^{\alpha\beta}(\phi^{\varepsilon}_{k})^{\alpha}_i(\eta_{\varepsilon}w)^{\beta}_j~dy-\sigma^{\varepsilon}_k
\int_{\Omega_{\varepsilon}}(\phi^{\varepsilon}_{k})^{\alpha}(\eta_{\varepsilon}w)^{\alpha}~dy\right|\\
&\hspace{3cm}=\left|I+II\right|\\&\hspace{3.5cm}+\left|III+IV\right|.
\end{align*}
First, since $\phi^{\varepsilon}_{k}$ is an eigenfunction with
eigenvalue $\sigma^{\varepsilon}_k$, we have that $III+IV=0$. Also,
by H\"{o}lder's inequality and Poincar\'{e}'s inequality, we have
\begin{align*}
|I+II|&\leq
\frac{C}{\varepsilon}\|\phi^{\varepsilon}_{k}\|_{L^2(B_{\varepsilon}\cup
\widetilde{B}_{\varepsilon})}\left(\|\nabla w\|_{L^2(B_{\varepsilon}\cup
\widetilde{B_{\varepsilon}})}+\|w\|_{L^2(B_{\varepsilon}\cup
\widetilde{B_{\varepsilon}})}\right)\\
&\leq C\|\nabla \phi^{\varepsilon}_{k}\|_{L^2(B_{\varepsilon}\cup
\widetilde{B}_{\varepsilon})}\|w\|_1
\end{align*}
where we have used G{\aa}rding's inequality (\ref{eqn1111}) on the
last line for $w$. Thus, from H\"{o}lder's inequality and
Proposition \ref{prop3},
\begin{align*}
|I+II| &\leq
C\varepsilon^{\frac{n(\widetilde{p}-2)}{2\widetilde{p}}}\|\nabla
\phi^{\varepsilon}_{k}\|_{L^{\widetilde{p}}(\Omega_{\varepsilon})}\|w\|_1\\
&\leq
C\varepsilon^{\frac{n(\widetilde{p}-2)}{2\widetilde{p}}}\|w\|_1.
\end{align*}
Since $\sigma_k^{0}=\sigma_J^{0}$,  the proof of the lemma is concluded.
\end{proof}
If we choose an interval $I$ around $\sigma_k^{0}$ such that
$\sigma_k^{\varepsilon}\in I$, and let
$\displaystyle q(f,g)=\int_{\Omega_{0}}a_{ij}^{\alpha
\beta}f^{\alpha}_i g^{\beta}_j~dy$ and
$f=\eta_{\varepsilon}\phi^{\varepsilon}_{k}$, we aim to satisfy the
hypotheses for part 1 of Lemma \ref{lemma1}. In order to do this, we need $\|\eta_{\varepsilon}\phi^{\varepsilon}_k\|_{L^2(\Omega_0)}$ to be bounded away from 0. To achieve this, we
start with the following well-known proposition.
\begin{prop} \label{prop4} If $A$ is an $N\times N$ matrix and $v$ is a $N\times 1$ vector such that $Av=0$ and $\displaystyle \sum_{i\neq l}^N|A_{li}|<|A_{ll}|~~~\textrm{for each}~l=1,...,N$,
then $v=0$. \end{prop}The next proposition shows that the functions
$\{\eta_{\varepsilon}\phi^{\varepsilon}_k\}_{k=J}^{J+m_J-1}$ are
almost orthonormal.
\begin{prop} \label{prop5} For any $\varepsilon>0$ and $l,k\in \mathbb{N}$, $(J\leq l,k \leq J+m_J-1)$, if $\phi_k=\phi^{\varepsilon}_k$, we have the following estimates:
\begin{equation} \label{eqn17} \int_{\Omega_{\varepsilon}}\eta_{\varepsilon}^2|
\phi_k|^2~dy\geq
1-C\varepsilon^{\frac{d(\widetilde{p}-2)}{\widetilde{p}}}
\end{equation}
\begin{equation} \label{eqn18}  \left|\int_{\Omega_{\varepsilon}}\eta_{\varepsilon}^2
\phi_k \cdot \phi_l~dy\right|\leq
C\varepsilon^{\frac{d(\widetilde{p}-2)}{\widetilde{p}}}
~~\textrm{if}~k\neq l
\end{equation} where $C$ only depends on $|\Omega_0|$, $n$, and $\sigma^{0}_J$, and is independent of $\varepsilon$.\end{prop}
\begin{proof}
We start by showing $(\ref{eqn17})$. Since the eigenfunctions are
normalized, we obtain for each $k$,
\begin{align*}1-\int_{\Omega_{\varepsilon}}\eta_{\varepsilon}^2|\phi_k|^2~dy&=\int_{\Omega_{\varepsilon}}(1-\eta_{\varepsilon}^2)|\phi_k|^2~dy\\
&=\int_{T_{\varepsilon} \cup B_{\varepsilon}\cup
\widetilde{B}_{\varepsilon}}(1-\eta_{\varepsilon}^2)|\phi_k|^2~dy\\
&\leq \|\nabla
\phi_k\|_{L^{\widetilde{p}}(\Omega_{\varepsilon})}^2|T_{\varepsilon}
\cup B_{\varepsilon}\cup
\widetilde{B}_{\varepsilon}|^{\frac{\widetilde{p}-2}{\widetilde{p}}}\\
&\leq C_k\varepsilon^{\frac{d(\widetilde{p}-2)}{\widetilde{p}}}
\end{align*}
where, from (\ref{eqn15}), $C_k$ depends on $\sigma_k^{0}$. Again, since $\sigma_k^{0}=\sigma_J^{0}$, we have $(\ref{eqn17})$.

\noindent Next, to show $(\ref{eqn18})$, we have
\begin{align*}
\left|\int_{\Omega_{\varepsilon}}\eta_{\varepsilon}^2
\phi_k\cdot\phi_l~dy\right|&\leq\left|\int_{B_{\varepsilon}\cup
\widetilde{B}_{\varepsilon}}\eta_{\varepsilon}^2
\phi_k\cdot\phi_l~dy\right|+\left|\int_{\Omega_{0}\backslash
(B_{\varepsilon}\cup
\widetilde{B}_{\varepsilon})}\eta_{\varepsilon}^2 \phi_k\cdot\phi_l~dy\right|\\
&=\left|\int_{B_{\varepsilon}\cup
\widetilde{B}_{\varepsilon}}\eta_{\varepsilon}^2
\phi_k\cdot\phi_l~dy\right|+\left|\int_{\Omega_{0}\backslash
(B_{\varepsilon}\cup \widetilde{B}_{\varepsilon})}
\phi_k\cdot\phi_l~dy-\int_{\Omega_{\varepsilon}}
\phi_k\cdot\phi_l~dy\right|\\
&\leq \int_{B_{\varepsilon}\cup \widetilde{B}_{\varepsilon}}
|\phi_k\cdot\phi_l|~dy+\int_{T_{\varepsilon}\cup B_{\varepsilon}\cup
\widetilde{B}_{\varepsilon}}| \phi_k\cdot\phi_l|~dy,
\end{align*}
the second inequality following since the set of eigenfunctions form
an orthogonal set in $L^2(\Omega_{\varepsilon})$. So, next by
H\"{o}lder's inequality, we get
\begin{align*}
\left|\int_{\Omega_{\varepsilon}}\eta_{\varepsilon}^2
\phi_k\cdot\phi_l~dy\right|&\leq \left(\int_{B_{\varepsilon}\cup
\widetilde{B}_{\varepsilon}}
|\phi_k|^2~dy\right)^{\frac{1}{2}}\left(\int_{B_{\varepsilon}\cup
\widetilde{B}_{\varepsilon}}
|\phi_l|^2~dy\right)^{\frac{1}{2}}\\&\hspace{.5cm}+\left(\int_{T_{\varepsilon}\cup
B_{\varepsilon}\cup \widetilde{B}_{\varepsilon}}
|\phi_k|^2~dy\right)^{\frac{1}{2}}\left(\int_{T_{\varepsilon}\cup
B_{\varepsilon}\cup \widetilde{B}_{\varepsilon}} |\phi_l|^2~dy\right)^{\frac{1}{2}}\\
&= I+II.
\end{align*}
Now, from Poincar\'{e}'s inequality and $(\ref{eqn15})$, we get
\begin{align*}
I&\leq\left[\left(\int_{B_{\varepsilon}\cup
\widetilde{B}_{\varepsilon}}
|\phi_k|^{\widetilde{p}}~dy\right)^{\frac{2}{\widetilde{p}}}|B_{\varepsilon}\cup
\widetilde{B}_{\varepsilon}|^{\frac{\widetilde{p}-2}{\widetilde{p}}}\right]^{\frac{1}{2}}\left[\left(\int_{B_{\varepsilon}\cup
\widetilde{B}_{\varepsilon}}
|\phi_l|^{\widetilde{p}}~dy\right)^{\frac{2}{\widetilde{p}}}|B_{\varepsilon}\cup
\widetilde{B}_{\varepsilon}|^{\frac{\widetilde{p}-2}{\widetilde{p}}}\right]^{\frac{1}{2}}\\
&\leq\|\nabla \phi_k\|_{L^{\widetilde{p}}(\Omega_{\varepsilon})}\varepsilon^{\frac{n(\widetilde{p}-2)}{2\widetilde{p}}}\|\nabla\phi_l\|_{L^{\widetilde{p}}(\Omega_{\varepsilon})}\varepsilon^{\frac{n(\widetilde{p}-2)}{2\widetilde{p}}}    \\
&\leq
C_k\varepsilon^{\frac{n(\widetilde{p}-2)}{2\widetilde{p}}}C_l\varepsilon^{\frac{n(\widetilde{p}-2)}{2\widetilde{p}}}
\end{align*}
where $C_k$ again depends on $\sigma_k^{0}$ and $C_l$ depends on
$\sigma_l^{0}$. Thus, we have
\begin{equation} \label{meqn10} I\leq
C\varepsilon^{\frac{n(\widetilde{p}-2)}{\widetilde{p}}}\end{equation}
where $C$ depends only on $|\Omega_0|$, $n$, and $\sigma_J^0$.
Similarly,
\begin{equation} \label{meqn11}II\leq C\varepsilon^{\frac{d(\widetilde{p}-2)}{\widetilde{p}}},\end{equation} so that the proposition is proved.
\end{proof}
Note that with the aid of Lemma \ref{lemma2} and Proposition \ref{prop5}, if $\varepsilon$ is small enough, we have satisfied the hypotheses for part 1 of Lemma \ref{lemma1} with $\displaystyle q(f,g)=$ \newline $\int_{\Omega_{0}}a_{ij}^{\alpha
\beta}f^{\alpha}_i g^{\beta}_j~dy$ and
$f=\eta_{\varepsilon}\phi^{\varepsilon}_{k}$. Here, we relabel $\varepsilon_1$ to be small enough to achieve this for any $\varepsilon\leq \varepsilon_1$, and note that $\varepsilon_1$ only depends on fixed parameters. To satisfy the hypotheses for part 2 of Lemma \ref{lemma1}, we need
an orthonormal basis. The next proposition shows that for small
$\varepsilon$, we have a basis.

\begin{prop} \label{prop6} The set $\{\eta_{\varepsilon}
\phi_k^{\varepsilon}\}_{k=J}^{N}$ forms a linearly independent set
for any $N\geq J$, for $0<\varepsilon\leq \varepsilon_0(N)$, where
$\varepsilon_0(N)$ depends on $N$.\end{prop}
\begin{proof}
Assume
$C_J\eta_{\varepsilon}\phi_J^{\varepsilon}+...+C_{N}\eta_{\varepsilon}\phi_{N}^{\varepsilon}=0$.
Then, multiplying this equation by
$\eta_{\varepsilon}\phi_l^{\varepsilon}$, we obtain
$$\sum_{k=J}^{N}C_k\langle\eta_{\varepsilon}\phi_k^{\varepsilon},\eta_{\varepsilon}\phi_l^{\varepsilon}\rangle_{L^2(\Omega_{\varepsilon})}=0,~~l=J,...,N.$$
So, if
$A_{kl}=\langle\eta_{\varepsilon}\phi_k^{\varepsilon},\eta_{\varepsilon}\phi_l^{\varepsilon}\rangle_{L^2(\Omega_{\varepsilon})}$,
we obtain by $(\ref{eqn17})$ and $(\ref{eqn18})$ that
\begin{align*}|A_{kk}|&\geq 1-C\varepsilon^{\frac{d(\widetilde{p}-2)}{\widetilde{p}}}\\
&>C\varepsilon^{\frac{d(\widetilde{p}-2)}{\widetilde{p}}}\\
&\geq \sum_{\substack{k=J\\k\neq l}}^{N}|A_{kl}|
\end{align*}
if $\varepsilon \leq \varepsilon(N)$, where $\varepsilon(N)$ depends on $N$ due to applying $(\ref{eqn18})$ $N-J$ times. Thus, we may use Proposition
$\ref{prop4}$ to see that by setting $C=(C_J,...,C_{N})^t$, we have
$C=0$, so that the proposition is proved.
\end{proof}

Now we define $\displaystyle
J_0:L^2(\Omega_{\varepsilon})\rightarrow L^2(\Omega_0)$ to be given
by $J_0f=\eta_{\varepsilon}f$, and similarly, we define
$\displaystyle J_{\varepsilon}:L^2(\Omega_0)\rightarrow
L^2(\Omega_{\varepsilon})$ to be such that
$$J_{\varepsilon}f(x)=\begin{cases}f(x),~~~~&\textrm{if}~~~~x\in
\Omega_0\\0,~~~~&\textrm{if}~~~~x\in\Omega_{\varepsilon} \backslash
\Omega_0. \end{cases}$$ Let
$I=\left(\sigma_J^0-M\varepsilon^{\frac{n(\widetilde{p}-2)}{4\widetilde{p}}},\frac{\sigma_J^0+\sigma_{J+m_J}^0}{2}\right)$
for $M>0$ to be chosen later. Also, let $\Pi$ be the projector onto
the space spanned by the eigenfunctions corresponding to the
eigenvalues, $\{ \sigma_k^{\varepsilon}\}_{k=J}^{N}$, in $I$.  We first consider $\varepsilon=\varepsilon_1$. By Proposition \ref{prop2},
we may choose $M=M(\varepsilon_1)$ so that $\sigma_k^{\varepsilon}$ is in $I$ for
$J\leq k\leq N$, where $N\geq J+m_{J}-1$, and where $N$ depends on
$\varepsilon_1$. We next note that as $\varepsilon$ gets smaller,
we may choose $M=M(\varepsilon)$ so that the set of eigenvalues in $I$, $\{
\sigma_k^{\varepsilon}\}_{k=J}^{N_0}$, will have index $N_0$ in the
range $J+m_{J}-1 \leq N_0\leq N$ since our family
$\{\Omega_{\varepsilon}\}$ is nested. Our aim is to show that for
$\varepsilon$ small, $N_0=J+m_J-1$.

We apply Proposition \ref{prop6} to get the existence of
$\varepsilon_0(N)\leq \varepsilon_1$ so that $\{\eta_{\varepsilon}
\phi_k^{\varepsilon}\}_{k=J}^{N_0}$ is a linearly independent set for
$\varepsilon\leq \varepsilon_0(N)$ and for any $N_0$ in the range $J+m_{J}-1 \leq N_0\leq N$. Then, we choose $M=M(\varepsilon(N))$ so that $\{\eta_{\varepsilon}
\phi_k^{\varepsilon}\}_{k=J}^{N_0}$ is also a basis for the range of $J_0\Pi J_{\varepsilon}$. 
Thus, we may apply
the Gram-Schmidt process to this basis. That is, define
\begin{align*}
&f_J=\eta_{\varepsilon}\phi_J^{\varepsilon}\\
&\vdots\\
&f_k=\eta_{\varepsilon}\phi_k^{\varepsilon}-\frac{\langle
\eta_{\varepsilon}\phi_k^{\varepsilon},f_J
\rangle}{\|f_J\|^2}f_J-...-\frac{\langle
\eta_{\varepsilon}\phi_k^{\varepsilon},f_{k-1}
\rangle}{\|f_{k-1}\|^2}f_{k-1}\\
\vdots
\end{align*}
We have the following lemma:
\begin{lemma} \label{lemma4.3} Let $I$ be as defined above. For each $k$, $J\leq k\leq J+m_J-1$, we have $\displaystyle\|\Pi_{I^c}(f_k)\|_1\leq
\frac{C\varepsilon^{\frac{d(\widetilde{p}-2)}{4\widetilde{p}}}}{M},$
for $\varepsilon\leq \varepsilon(N)$, and where $M$ only depends
on $\sigma_J^{0}$, $\sigma_{J-1}^{0}$, and $\varepsilon(N)$. 
\end{lemma}
\begin{proof}
Following the previous arguments, when $\varepsilon=\varepsilon(N)$, we find $M=M(\varepsilon(N))$ so that $\{\eta_{\varepsilon}
\phi_k^{\varepsilon}\}_{k=J}^{N_0}$ is a basis for the range of $J_0\Pi J_{\varepsilon}$, and then apply the Gram-Schmidt process to this basis. We
note the dependence on $\sigma_J^{0}$ and $\sigma_{J-1}^{0}$ is so
that we only have 1 eigenvalue (with respect to $\Omega_0$) in $I$. So, defining
$q(f,g)=\int_{\Omega_{0}}a_{ij}^{\alpha \beta}f^{\alpha}_i
g^{\beta}_j~dy$, we may apply Lemma \ref{lemma2}, Proposition \ref{prop5}, and then Lemma
\ref{lemma1} (part 1) to obtain
$$\|\Pi_{I^c}(f_J)\|_1\leq
\frac{C\varepsilon(N)^{\frac{d(\widetilde{p}-2)}{4\widetilde{p}}}}{M(\varepsilon(N))}$$
where $C$ depends on $|\Omega_0|$, $n$, $\sigma_J^{0}$, and
$\sigma_{J+m_J}^{0}$. Then, from Proposition \ref{prop5}, Lemma
\ref{lemma2}, and properties of the norm, we get the result for $\varepsilon(N)$ and $J\leq k \leq J+m_J-1$. Then, for $\varepsilon\leq \varepsilon(N)$, we may repeat this argument to get the result with $\varepsilon(N)$ replaced with $\varepsilon$ and $M(\varepsilon(N)$ replaced with $M(\varepsilon)$. But, since $M(\varepsilon(N))\leq M(\varepsilon)$, we obtain the desired result for $\varepsilon \leq \varepsilon(N)$.
\end{proof}
We now let $E=\textrm{span}\{\phi^{\varepsilon}_k\}_{k=J}^{J+m_J-1}$. Also, let $\Pi_I$ be the spectral projector corresponding to the eigenvalue $\sigma_J^0$ and $\Pi_E$ be the spectral projector onto $E$.
\begin{corollary} \label{cor1} We have $\displaystyle\|\Pi_I-J_0\Pi_E
J_{\varepsilon}\|_{{\cal L}\{L^2(\Omega_0)\}}\leq
\frac{C\varepsilon^{\frac{d(\widetilde{p}-2)}{4\widetilde{p}}}}{M}$,
for $\varepsilon\leq \varepsilon(N)$, where $M$ only depends on
$\sigma_J^{0}$, $\sigma_{J-1}^{0}$, and $\varepsilon(N)$. Consequently, for some $\varepsilon(J)$, $N_0=J+m_J-1$ when $\varepsilon \leq \varepsilon(J)$.
\end{corollary}
\begin{proof}
Again, we first show for $\varepsilon=\varepsilon(N)$. Normalize the $f_k$'s and observe that
$\displaystyle\frac{1}{\|f_k\|}\leq
\frac{1}{1-C\varepsilon(N)^{\frac{d(\widetilde{p}-2)}{2\widetilde{p}}}}$.
Then apply Lemma \ref{lemma1} (part 2) to the normalized functions. Then for general $\varepsilon \leq \varepsilon(N)$, we note that since Lemma \ref{lemma4.3} is true with a uniform $M$, we obtain $\displaystyle\|\Pi_I-J_0\Pi_E
J_{\varepsilon}\|_{{\cal L}\{L^2(\Omega_0)\}}\leq
\frac{C\varepsilon^{\frac{d(\widetilde{p}-2)}{4\widetilde{p}}}}{M}$. We next note that if $N_0>J+m_J-1$ for all $\varepsilon \leq \varepsilon(N)$, then we may find another projector $\Pi_A$ so that $\displaystyle\|\Pi_I-J_0\Pi_A
J_{\varepsilon}\|_{{\cal L}\{L^2(\Omega_0)\}}\leq
\frac{C\varepsilon^{\frac{d(\widetilde{p}-2)}{4\widetilde{p}}}}{M}$. But this would mean $\displaystyle\|J_0\Pi_EJ_{\varepsilon}-J_0\Pi_A
J_{\varepsilon}\|_{{\cal L}\{L^2(\Omega_0)\}}\leq
\frac{C\varepsilon^{\frac{d(\widetilde{p}-2)}{4\widetilde{p}}}}{M}$. Therefore, for some $\varepsilon(J)$, $N_0=J+m_J-1$ when $\varepsilon \leq \varepsilon(J)$.
\end{proof}

\vspace{.5in}

\begin{proof}[Proof of Theorem \ref{thm1}]
We first prove for $J=1$. By Corollary \ref{cor1}, for $\varepsilon \leq \varepsilon(1)$, we obtain $m_1=N_0$. This implies
that $|\sigma_k^{\varepsilon}-\sigma_k^0|\leq C\varepsilon^{\frac{d(\widetilde{p}-2)}{4\widetilde{p}}}$ only for $k$, $1\leq k\leq m_1$,
and hence, the result for $J=1$.   The result for $J=1$ implies that
not only may we choose $M$ so that all eigenvalues
$\{\sigma_k^{\varepsilon}\}_{k=m_1+1}^{m_1+m_2}$ are in the interval
corresponding to the next highest eigenvalue $\sigma_{m_1+1}^{0}$,
but also that $\sigma_1^{0}$ is not in this interval. Thus, we apply
the same reasoning here to get the result for $\sigma_{m_1+1}^{0}$.
Then, by an induction argument, we get the result for each $J\in
\mathbb{N}$, satisfying $\sigma_J^{0}>\sigma_{J-1}^{0}$. We note here that since C depends on $\varepsilon(J)$, it depends on the multiplicity $J$.
\end{proof}

We note that this paper introduces the use of $L^p$-estimates obtained by the reverse H\"older technique to the study of spectral problems for elliptic operators. Thus, this technique may be useful in studying spectral problems in situations where we do not know if higher regularity of solutions is true. We close by listing some open problems.

\begin{itemize}
\item If we have some additional regularity on the domain, can we use the methods
from this work to get convergence of Neumann eigenvalues for general
elliptic systems?
\item For elliptic systems on a symmetric dumbbell region with a straight tube, can we achieve upper and lower
bounds on the splitting between the smallest eigenvalues?
\item Can we investigate this problem further to see if a better
rate of convergence exists?
\end{itemize}
\vspace{1in}

\noindent \textbf{Acknowledgments:} The author thanks Russell Brown
for his valuable discussions and suggestions. The author also thanks the referee for his or her helpful comments.

\bibliographystyle{model1-num-names}
\bibliography{sample}
\today
\end{document}